\documentclass[11pt,a4paper]{amsart}
\usepackage{amssymb,xspace}
\usepackage{amstext}
\theoremstyle{plain}
\usepackage{amsbsy,amssymb,amsfonts,latexsym}

\usepackage{enumerate}
\usepackage{graphics}

\date{\today}
\title{Multifractal analysis of the divergence of Fourier series: the extreme cases}
\author{Fr\'ed\'eric Bayart, Yanick Heurteaux}
\address{
Clermont Universit\'e, Universit\'e Blaise Pascal, Laboratoire de Math\'ematiques, 
BP 10448, F-63000 CLERMONT-FERRAND -
CNRS, UMR 6620, Laboratoire de Math\'ematiques, F-63177 AUBIERE
}
\email{Frederic.Bayart@math.univ-bpclermont.fr, Yanick.Heurteaux@math.univ-bpclermont.fr}



\newcommand{\veps}{\varepsilon}

\def\RR{\mathbb R}

\def\NN{\mathbb N}
\def\ZZ{\mathbb Z}
\def\TT{\mathbb T}
\def\DD{\mathbb D}

\newcommand{\pss}[2]{\ensuremath{{\langle #1,#2\rangle}}}

\newtheorem{theorem}{Theorem}[section]

\newtheorem{lemma}[theorem]{Lemma}

\newtheorem{proposition}[theorem]{Proposition}

\newtheorem{corollary}[theorem]{Corollary}

{\theoremstyle{definition}}
{\theoremstyle{definition}}

{\theoremstyle{definition}}

{\theoremstyle{definition}\newtheorem{definition}[theorem]{Definition}}

{\theoremstyle{definition}}

{\theoremstyle{definition}}

\newenvironment{rem}
{ \noindent {\it Remark\/}: }
{\null \hfill \par\medskip}
\newenvironment{rems}
{ \noindent {\it Remarks\/}: }
{\null \hfill \par\medskip}

\begin{document}

\begin{abstract}
We study the size, in terms of the Hausdorff dimension, of the subsets of $\mathbb T$
such that the Fourier series of a generic function in  $L^1(\TT)$,  $L^p(\TT)$ or in $\mathcal C(\mathbb T)$ may behave badly. Genericity is related to the Baire category theorem or to the notion of prevalence. This paper is a continuation of \cite{BH11}.
\end{abstract}

\maketitle

\section{Introduction}

This paper, which can be seen as a continuation of \cite{BH11}, 
deals with the divergence of Fourier series of functions 
in $L^p(\TT)$, $p\geq 1$, where $\mathbb T=\mathbb R/\mathbb Z$, or in 
$\mathcal C(\TT)$, from the multifractal point of view.
More precisely, let $f$ be in $L^p(\mathbb T)$, or in $\mathcal C(\TT)$, and
let $(S_n f)_{n\ge 0}$ the sequence of partial sums of its Fourier series.
We are interested in the size of the sets of the real numbers $x$ such that 
$(S_nf(x))_{n\ge 0}$ 
diverges with a prescribed growth. 

We will  measure the size of subsets of $\TT$
using the Hausdorff dimension. Let us recall the relevant definitions (we refer to 
\cite{Falc} and to \cite{Mat95} for more on this subject). 
If $\phi:\mathbb R_+\to\mathbb R_+$ is a nondecreasing continuous function
satisfying $\phi(0)=0$ ($\phi$ is called a \emph{dimension function} or a 
\emph{gauge function}),
the $\phi$-Hausdorff outer measure of a set $E\subset \mathbb R^d$ is 
$$\mathcal H^{\phi}(E)=\lim_{\veps\to 0}\inf_{r\in R_\veps(E)}\sum_{B\in r}\phi(|B|),$$
where $R_\veps(E)$ is the set of (countable) coverings of $E$ with balls $B$ of diameter 
$|B|\leq\veps$. 
When $\phi_s(x)=x^s$, we write for short $\mathcal H^s$ instead of $\mathcal H^{\phi_s}$. 
The Hausdorff dimension of a set $E$
is defined by
$$\dim_{\mathcal H}(E):=\sup\{s>0; \mathcal H^s (E)>0\}=\inf\{s>0;
\ \mathcal H^s(E)=0\}.$$

The first result studying the Hausdorff dimension of the divergence sets of
 Fourier series is due to J-M. Aubry \cite{Aub06}.

\begin{theorem}\label{THMAUBRY}
Let $f\in L^p(\mathbb T)$, $1<p<+\infty$. If $\beta\geq 0$, define
$$\mathcal E(\beta,f)=\left\{x\in\mathbb T;\ 
\limsup_{n\to+\infty}n^{-\beta}|S_nf(x)|>0\right\}.$$
Then $\dim_\mathcal{H}\big(\mathcal E(\beta,f)\big)\leq 1-\beta p$. Conversely, 
given a set $E$ such that $\dim_\mathcal{H}(E)<1-\beta p$, there exists
a function $f\in L^p(\mathbb T)$ such that, for any $x\in E$,
$\displaystyle\limsup_{n\to +\infty}n^{-\beta}|S_nf(x)|=+\infty$.
\end{theorem}

\medskip

This result motivated us to introduce in \cite{BH11} the notion of divergence index. 
For a given function $f\in L^p(\TT)$ and a given point $x_0\in \TT$, we can define
$\beta(x_0)$ as the infimum of the nonnegative real numbers
$\beta$ such that $|S_nf(x_0)|=O(n^\beta)$. The real number $\beta(x_0)$ will
be called the \emph{divergence index} of the Fourier series of $f$ 
at point $x_0$. It is well-known that, for any function $f\in L^p(\TT)$ 
($1\leq p<+\infty$) and any point $x_0\in\TT$, 
$0\le \beta(x_0)\le 1/p$ (see \cite{Zyg}). Moreover, when $p>1$, Carleson's theorem implies 
that $\beta(x_0)=0$
almost surely. In \cite{BH11}, we gave precise estimates on the size of the level sets of the function
$\beta$. These are defined as

\begin{eqnarray*}
E(\beta,f)&=&\left\{x\in\mathbb T;\ \beta(x)=\beta\right\}\\
&=&\left\{x\in\mathbb T;\ \limsup_{n\to+\infty}\frac{\log |S_nf(x)|}{\log
    n}=\beta\right\}.
\end{eqnarray*}
\begin{theorem}[\cite{BH11}]\label{THMBHLP}
Let $1<p<+\infty$. For quasi-all functions $f\in L^p(\TT)$, for any $\beta\in[0,1/p]$, 
$\dim_\mathcal{H}\big(E(\beta,f)\big)=1-\beta p$.
\end{theorem}
The terminology "quasi-all" used here is relative to the Baire category theorem. It means that this property is true for a residual
set of functions in $L^p(\TT)$.

\medskip

In the case of continuous functions, the situation breaks down
dramatically. If $(D_n)_{n\ge 0}$ denotes the Dirichlet kernel, 
we can first observe that, when $f\in\mathcal C(\TT)$, 
$$\|S_n f\|_\infty\leq \|D_n\|_1\|f\|_\infty\leq C\|f\|_\infty\log n.$$
This motivated us in \cite{BH11} to introduce the following level sets:
\begin{eqnarray*}
\mathcal F(\beta,f)&=&\left\{x\in\mathbb T;\ \limsup_{n\to+\infty}\,(\log n)^{-\beta}
|S_nf(x)|>0\right\}\\
F(\beta,f)&=&\left\{x\in\mathbb T;\ \limsup_{n\to+\infty}\frac{\log |S_nf(x)|}{\log\log n}=\beta\right\}.
\end{eqnarray*}
Whereas, on $L^p(\TT)$, $1<p<+\infty$, the divergence index takes its biggest value  ($\beta(x)=1/p$) on small sets, 
this is far from being the case on $\mathcal C(\TT)$, as the following very surprizing 
result indicates.
\begin{theorem}[\cite{BH11}]\label{THMBHCT}
For quasi-all functions $f\in\mathcal C(\TT)$, for any $\beta\in[0,1]$, $F(\beta,f)$ is 
non-empty
and has Hausdorff dimension 1.
\end{theorem}

\medskip

However, several questions were left open in \cite{BH11}.

\medskip
{\bf Question 1: what happens on $L^1(\TT)$?} In view of the differences between 
$L^p(\TT)$, $p\in(1,+\infty)$, and $\mathcal C(\TT)$, 
it seems \emph{a priori} not clear what situation should be expected on $L^1(\TT)$. 
Moreover, Carleson's theorem is false on $L^1(\TT)$
and Kolmogorov Theorem ensures that there exist functions in $L^1(\TT)$ with everywhere 
divergent Fourier series. \\
The proof of Theorem \ref{THMBHLP} proceeds in two steps. In a first time, we build a 
residual set of functions in $L^p(\TT)$
such that, if $f$ lies in this residual set and if $0\le\beta\le1/p$, 
$\dim_{\mathcal H}(E(\beta,f))\geq 1-\beta p$. In a second time, we use Theorem 
\ref{THMAUBRY} to conclude that
necessarily $\dim_{\mathcal H}(E(\beta,f))= 1-\beta p$. The first step works as well 
in $L^1(\TT)$ and the trouble comes from Aubry's result,
which uses the Carleson Hunt maximal inequality. In Section \ref{SECL1}, we succeed to 
overcome this difficulty by proving a (very weak!) version of Carleson's
maximal inequality in $L^1(\TT)$ which is sufficient to prove the analogue of Theorem 
\ref{THMAUBRY}. Thus, we will show that
\begin{theorem}\label{THML1}
For quasi-all functions $f\in L^1(\TT)$, for any $\beta\in[0,1]$, 
$$\dim_\mathcal{H}\big(E(\beta,f)\big)=1-\beta.$$
\end{theorem}
\medskip

{\bf Question 2: what about the size of the set of multifractal functions?} Theorem
\ref{THMBHLP} and Theorem \ref{THML1} say that, in $L^p(\TT)$ ($p\ge 1$), the set of 
multifractal functions
is big in a topological sense. One can ask if it remains big for other points of view. 
We deal here with an infinite-dimensional version
of the notion of "almost-everywhere". This notion, called \emph{prevalence}, has been 
introduced by J. Christensen in \cite{Chr72} and has been
widely studied since then. In multifractal analysis, some properties which are true on 
a dense $G_\delta$-set are also prevalent 
(see for instance \cite{FJK} or \cite{FJ06}), whereas some are not (see for instance 
\cite{FJK} or \cite{Ol10}). This motivated us to examine Theorem \ref{THMBHLP} and Theorem \ref{THML1} under this point of view.
\begin{definition}
Let $E$ be a complete metric vector space. A Borel set $A\subset E$ is called 
\emph{Haar-null} if there exists a compactly supported 
probability measure $\mu$ such that, for any $x\in E$, $\mu(x+A)=0$. If this property 
holds, the measure $\mu$
is said to be \emph{transverse} to $A$.\\
A subset of $E$ is called \emph{Haar-null} if it is contained in a Haar-null Borel set. 
The complement of a Haar-null set
is called a \emph{prevalent} set.
\end{definition}
The following results enumerate important properties of prevalence and show that this 
notion supplies a natural generalization
of "almost every" in infinite-dimensional spaces:
\begin{itemize}
\item If $A$ is Haar-null, then $x+A$ is Haar-null for every $x\in E$.
\item If $\dim(E)<+\infty$, $A$ is Haar-null if and only if it is negligible
  with respect to the  Lebesgue measure.
\item Prevalent sets are dense.
\item The intersection of a countable collection of prevalent sets is prevalent.
\item If $\dim(E)=+\infty$, compacts subsets of $E$ are Haar-null.
\end{itemize}

\smallskip

In Section \ref{SECPREVALENCE}, we will prove the following result.
\begin{theorem}\label{THMPREVALENCE}
Let $1\leq p<+\infty$. The set of functions $f\in L^p(\TT)$ such that, for any 
$\beta\in[0,1/p]$, $\dim_{\mathcal H}(E(\beta,f))=1-\beta p$,
is prevalent.
\end{theorem}
Thus, almost every function in $L^p(\TT)$ is multifractal with respect to the
summation of its Fourier series.

\medskip
{\bf Question 3: can we say more on $\mathcal C(\TT)$?}
Theorem \ref{THMBHCT} implies that there exists a residual subset
$A\subset\mathcal C(\TT)$ such that, if $f\in A $ and if 
$\beta<1$, one can find a set $E\subset\TT$ with Hausdorff dimension 1 such that
\begin{eqnarray}\label{EQCT1}
\limsup_{n\to +\infty} \frac{|S_nf(x)|}{(\log n)^\beta}=+\infty\textrm{ for any }x\in E. 
\end{eqnarray}
On the other hand, we know that, for any fixed $f\in\mathcal C(\TT)$, $\|S_n
f\|_\infty$ is negligible compared to $\log n$ 
and that, conversely, given any sequence $(\delta_n)_{n\ge 2}$ of positive real numbers 
going to zero,
we can find $f\in\mathcal C(\TT)$ such that 
$$\limsup_{n\to +\infty} \frac{|S_nf(0)|}{\delta_n\log n}=+\infty.$$
These statements can be found for example in \cite{Zyg}. It seems then natural to ask whereas this property 
can be ensured in a set with Hausdorff dimension equal to 1
( (\ref{EQCT1}) meaning that this is true when $\delta_n=(\log n)^{\beta-1}$, 
$0<\beta<1$). 
This is indeed true.
\begin{theorem}\label{THMCT}
Let $(\delta_n)_{n\ge 2}$ be a sequence of positive real numbers going to zero. 
For quasi-all functions $f\in\mathcal C(\TT)$,
there exists $E\subset\TT$ with Hausdorff dimension 1 such that, for any $x\in E$, 
$$\limsup_{n\to +\infty} \frac{|S_nf(x)|}{\delta_n\log n}=+\infty.$$
\end{theorem}

The same result also holds in a prevalent subset of $\mathcal C(\TT)$.

\begin{theorem}\label{THMCTPREVALENT}
Let $(\delta_n)_{n\ge 2}$ be a sequence of positive real numbers going to zero. 
For almost every function $f\in\mathcal C(\TT)$,
there exists $E\subset\TT$ with Hausdorff dimension 1 such that, for any $x\in E$, 
$$\limsup_{n\to +\infty} \frac{|S_nf(x)|}{\delta_n\log n}=+\infty.$$
\end{theorem}
The proof of Theorems \ref{THMCT} and \ref{THMCTPREVALENT} are proposed in Section \ref{SECCT}.

\section{Multifractal analysis of the divergence of Fourier series in $L^1(\TT)$}
\label{SECL1}

We first recall some basic facts on Fourier series and Fourier transforms in $L^p$. 
Let $\xi\in\mathbb R$ and $e_\xi:t\mapsto e^{2\pi i \xi t}$. 
The Fourier transform of $f\in L^1(\RR)$ is the continuous function 
$$\hat f:\xi\mapsto \int_\mathbb R f(x)\bar{e_\xi}(x)dx.$$ 
The operator makes also
sense in the space $L^p(\RR)$ when $1\le p<+\infty$. In that case, $\hat f\in
L^q(\RR)$ where $\frac1p+\frac1q=1$. In $L^p(\RR)$ we can define the band-limiting 
operator $S_n$ by 
$$\widehat{S_n f}=\mathbf 1_{[-n,n]}\hat f.$$ It is well known that, on
$L^p(\RR)$, the projections $(S_n)_{n\geq 0}$ are uniformly bounded; this is 
the Riesz theorem. This is not the case on $L^1(\RR)$. However, there exists some 
absolute constant $C>0$ such that, for any $n\geq 2$ and any $f\in L^1(\RR)$, 
$$\|S_n f\|_1\leq C\log n \|f\|_1.$$

A function $g\in L^1(\TT)$ 
is identified to a $1$-periodic function on $\RR$. Its Fourier transform is the 
tempered distribution 
$$\hat g=\sum_{k\in\ZZ} \pss{g}{e_k}\delta_k,$$
where $\pss{g}{e_k}=\int_{\TT}g(t)\bar{e_k}(t)dt$ are the Fourier coefficients of $g$  
and $\delta_k$ denotes the Dirac
mass at point $k$. 
If $g\in L^1(\TT)$, the band limiting operator corresponds to taking the partial sum of 
the Fourier series,
$$S_n g:t\mapsto \sum_{k=-n}^n \pss g{e_k}e_k(t).$$
We can also write $S_ng=D_n*g$ where 
$$D_n(t)=\sum_{k=-n}^n e_k(t)=\frac{\sin\big(\pi(2n+1)t\big)}{\sin(\pi t)}$$
is the Dirichlet kernel and the Riesz theorem always occurs in this context.

Let us also recall the definition of  $\sigma_n g$, the $n$-th F\'ejer sum of $g$, namely
$$\sigma_n(g)=\frac1n\big(S_0g+\dots+S_{n-1}g\big).$$

\medskip
We write $\mathcal E_n(\TT):=S_n(L^1(\TT))$ the set of trigonometric polynomials of 
degree less
than $n$ and $\mathcal E_n(\RR):=S_n(L^1(\RR))$. The classical
Nikolsky inequality (see for example \cite{Nyk}) says that if $P\in\mathcal E_n(\TT)$ 
or $P\in\mathcal E_n(\RR)$ and $1\leq p\leq q\leq \infty$, then 
$$\|P\|_q\leq n^{\frac1p-\frac1q}\|P\|_p.$$

Our first lemma will be helpful to control a function which is locally a Dirichlet 
kernel.

\begin{lemma}\label{LEML1-1}
There exists a constant $A>0$ such that, for any $N\geq 2$, for any measurable function  
$n:\TT\to\{1,\dots,N\}$, for any $t\in\TT$, then
$$\int_{\TT}|D_{n(x)}(x-t)|dx\leq A\log N.$$
\end{lemma}
\begin{proof}
It is obvious from the above expression of $D_n$ that, if $k\leq N$ and if 
$u\in [-1/2,1/2]$, 
$$|D_k(u)|\leq\left\{
\begin{array}{c}
CN\\
\frac{C}{|u|}
\end{array}\right.$$
for some absolute constant $C>0$. We then split the integral into two parts:
$$\int_{|x-t|\leq1/N}|D_{n(x)}(x-t)|dx\leq 2CN\frac1N$$
and
$$\int_{1/N<|x-t|\le1/2}|D_{n(x)}(x-t)|dx\leq C\int_{1/N<|x-t|\le1/2}\frac{dx}{|x-t|}
\leq 2C\log N.$$
\end{proof}
Writing $S_{n(x)}f(x)=(f\star D_{n(x)})(x)$ and using Fubini's theorem, it is 
straightforward to deduce the following
inequality on partial sums of Fourier series of $L^1$-functions.
\begin{lemma}\label{LEML1-2}
There exists a constant $A>0$ such that, for any $N\geq 2$, for any measurable
function $n:\TT\to\{1,\dots,N\}$, for any $f\in L^1(\TT)$,
then 
$$\int_{\TT}|S_{n(x)}f(x)|dx\leq A\log N\,\|f\|_1.$$
\end{lemma}
We are now ready to prove the following weak version of the maximal inequality of 
Carleson and Hunt, on $L^1(\TT)$.
\begin{corollary}\label{CORL1-1}
Let $\alpha>0$. There exists $C:=C_\alpha>0$ such that, for any $f\in L^1(\TT)$, 
$$\int_{\TT}\sup_{n\geq 2}\frac{|S_n f(x)|}{(\log n)^{1+\alpha}}dx\leq C\,\|f\|_1.$$
\end{corollary}
\begin{proof}
Using the monotone convergence theorem, we first observe that
 it is sufficient to prove that, for any $N\geq 2$,
\begin{eqnarray}\label{EQCORL1-1}
\int_{\TT}\sup_{2\leq n\leq N}\frac{|S_n f(x)|}{(\log n)^{1+\alpha}}dx\leq C\,\|f\|_1
\end{eqnarray}
where, of course, $C$ does not depend on $N$. Now, we take a measurable
function $n:\TT\to\NN\backslash\{0,1\}$ not necessarily bounded, 
and observe that
(\ref{EQCORL1-1}) will be proved if we are able to show that
$$\int_{\TT}\frac{|S_{n(x)}f(x)|}{(\log n(x))^{1+\alpha}}dx\leq C\,\|f\|_1$$
for some constant $C$ independent of the function $n$. 
If $k\geq 0$, let 
$$A_k=\{x\in\TT;\ 2^{2^k}\leq n(x)< 2^{2^{k+1}}\}.$$ 
Lemma \ref{LEML1-2} ensures that
\begin{eqnarray*}
\int_{\TT}\frac{|S_{n(x)}f(x)|}{(\log n(x))^{1+\alpha}}dx&=&\sum_{k\geq
  0}\int_{A_k}\frac{|S_{n(x)}f(x)|}{(\log n(x))^{1+\alpha}}dx\\
&\leq&\sum_{k\geq 0}\frac1{\left( 2^k\log 2\right)^{1+\alpha}}
\int_{A_k}|S_{n(x)}f(x)|dx\\
&\leq&\sum_{k\geq 0}C\frac{2^{k+1}\log2}{2^{k(1+\alpha)}(\log2)^{1+\alpha}}\,\|f\|_1\\
&=&C_\alpha\,\|f\|_1.
\end{eqnarray*}
\end{proof}

The following lemma is inspired by Aubry's paper. It means that, as soon as a 
trigonometric polynomial is large at some point $a\in\TT$,
it is also large in small intervals around $a$, with a rather good control of the 
$L^p$-norm.
\begin{lemma}\label{LEML1-3}
Let $p\geq 1$ and $\veps>0$. There exists $\delta>0$ such that, if $n$ is
large enough, if $P\in\mathcal E_n(\TT)$ and if $a\in\TT$ is such that 
$|P(a)|\ge\|P\|_p$, then, for any interval $I$ with center $a$ and with length 
$|I|\leq \frac 1n$,

$$\|P\|_{L^p(I)}\geq\delta |P(a)|\times|I|^{1/p}\times \left\{
\begin{array}{ll}
\displaystyle \frac{1}{(\log n)^{(1+\veps)/p}}&\textrm{provided }p>1\\
\displaystyle \frac{1}{(\log n)^{1+\veps}\log(1/|I|)}&\textrm{provided }p=1.
\end{array}
\right.$$
\end{lemma}
\begin{rems}

- Such a point $a$ does exist because $P$ is continuous.\\
- In fact, we will only need the lemma in the case $p=1$, but we give the
general case for completeness.
\end{rems}
\begin{proof}[Proof of Lemma \ref{LEML1-3}]
Without loss of generality, we may assume that $a=0$. The idea is to localize $P$ 
around $0$, and to use Nikolsky inequality 
to estimate the $L^p$-norm knowing the $L^\infty$-norm. 
Let $\gamma\in(0,1)$ such that $\gamma(1+\veps)>1$. We introduce a function $w$ with 
support in $[-1,1]$
satisfying $0\leq w\leq 1$, $w(0)=1$ and for which there exist two strictly positive 
constants $D$ and $E$ such that
$$\forall \xi\in\RR,\quad |\hat w(\xi)|\leq De^{-E|\xi|^\gamma}.$$
It is a classical result in Fourier analysis that such a function does exist (see e.g. 
\cite[Lemma 6]{Aub06}). We then set $w_I(x)=w(x/|I|)$. 
We decompose $Pw_I$ as $f_1+f_2$ with $f_1=S_{N}Pw_I$
and $N=[|I|^{-1}(\log n)^{1+\veps}]$, the integer part of $|I|^{-1}(\log n)^{1+\veps}$. 
On the one hand, if $p>1$ we get
\begin{eqnarray*}
\|f_1\|_{\infty}&\leq&N^{1/p}\|f_1\|_p\textrm{ (Nikolsky inequality)}\\
&\leq&C_p |I|^{-1/p}(\log n)^{(1+\veps)/ p}\|Pw_I\|_p\textrm{ (Riesz theorem)}\\
&\leq&C_p |I|^{-1/p}(\log n)^{(1+\veps)/ p}\|P\|_{L^p(I)}.
\end{eqnarray*}
When $p=1$, we have to add the norm of the Riesz projection,
and we get
$$\|f_1\|_\infty\leq C_1 |I|^{-1}(\log n)^{1+\veps}\log(1/|I|)\|P\|_{L^1(I)}.$$
On the other hand, we may write
\begin{eqnarray*}
\hat f_2(\xi)&=&\mathbf 1_{\{|\xi|>N\}}(\xi)(\hat{P}\star 
\hat{w_I})(\xi)\\
&=&\sum_{j=-n}^n \mathbf 1_{\{|\xi|>N\}}(\xi)\hat P(j)\hat{w_I}(\xi-j).
\end{eqnarray*}
Now, if $n$ is large enough  and  $j\leq n$, we have
\begin{eqnarray*}
\int_{|\xi|>N}|\hat{w_I}(\xi-j)|d\xi&\leq&\int_{|\xi|>\frac12|I|^{-1}(\log n)^{1+\veps}}|
\hat{w_I}(\xi)|d\xi\\
&=&\int_{|\xi|>\frac12(\log n)^{1+\veps}}|\hat{w}(\xi)|d\xi.
\end{eqnarray*}
Observe that
$$\int_A^{+\infty}e^{-E\xi^\gamma}d\xi=\frac1\gamma\int_{A^{\gamma}}^{+\infty}
e^{-Et}t^{1/\gamma-1}dt\le Ce^{-(E/2)A^{\gamma}}.$$
It follows easily that
$$\int_{|\xi|>N}|\hat{w_I}(\xi-j)|d\xi\leq Cn^{-2}$$

provided $n$ is large enough. This implies
\begin{eqnarray*}
\|f_2\|_\infty\leq \|\hat{f}_2\|_1&\leq&Cn^{-2}\sum_{j=-n}^n |\hat{P}(j)|\\
&\leq&Cn^{-2}(2n+1)\|P\|_1\\
&\leq&Cn^{-2}(2n+1)\|P\|_p\\
&\leq&\frac12\|P\|_p
\end{eqnarray*}
provided $n$ is large enough. If we recall that $|P(0)|\ge\|P\|_p$, we get
$$\|f_1\|_\infty\geq |P(0)|-\|f_2\|_\infty\geq \frac12|P(0)|$$
and the result follows from the above estimates of $\|f_1\|_\infty$.
\end{proof}
We can now conclude by proving the following proposition (Proposition \ref{PROPL1}) and its corollary  on
the Hausdorff dimension of $E(\beta,f)$ (Corollary \ref{CORAUBRYL1}). Recall that it is all that we need to obtain Theorem \ref{THML1}
since the construction done in \cite{BH11} is always true when $p=1$ and shows that there exists a residual set of 
functions $f\in L^1(\TT)$ with
$\dim_{\mathcal H}(E(\beta,f))\geq 1-\beta$ for any $\beta\in[0,1]$.
\begin{proposition}\label{PROPL1}
Let $f\in L^1(\TT)$ and  $\tau:(0,+\infty)\to(0,+\infty)$ be an increasing function. Define 
$$E(\tau,f):=\left\{x\in\TT;\ \limsup_{n\to +\infty}\frac{ |S_n f(x)|}{\tau(n)}=+\infty\right\}.$$ 
If $\nu>3$ and if $\phi$ is a dimension function  satisfying $c_1s\le\phi(s)\le c_2\frac{s\tau(s^{-1})}{\log(s^{-1})^\nu}$,
then $$\mathcal H^{\phi}(E(\tau,f))=0.$$
\end{proposition}

\begin{proof}
Let $M>0$ and $\veps=\nu-3$. Define
$$E_M(\tau,f)=\left\{x\in\TT;\ \limsup_{n\to +\infty}\frac{ |S_nf(x)|}{\tau(n)}>M\right\}.$$ 
If $x\in E_M(\tau,f)$, one can find $n_x$ as large as we want such that $|S_{n_x}f(x)|\geq M\tau(n_x)$.
Set $I_x=\left[x-\frac1{2n_x},x+\frac1{2n_x}\right]$ and observe that $\|S_{n_x}f\|_1\leq C(\log n_x)$. The hypothesis on the function $\tau$ implies that, if $n_x$ is large enough, 
$\|S_{n_x}f\|_1\le|S_{n_x}f(x)|$. We can then apply Lemma  \ref{LEML1-3} and we get
$$\|S_{n_x}f\|_{L^1(I_x)}\geq \delta\frac{M\tau(n_x)}{n_x(\log n_x)^{2+\veps/2}}.$$
$(I_x)_{x\in E_M(\tau,f)}$ is a covering of $E_M(\tau,f)$. We can extract a Vitali's covering, namely a countable
family of disjoint intervals $(I_i)_{i\in\NN}$, of length $1/n_i$, such that $E_M(\tau,f)\subset\bigcup_{i\in\NN}5B_i$.
Then, Corollary \ref{CORL1-1} implies
\begin{eqnarray*}
C\|f\|_1&\geq&\int_{\TT}\sup_{n\geq 2} \frac{|S_n f(x)|}{(\log n)^{1+\veps/2}}dx\\
&\geq&\sum_i \int_{I_i}\frac{|S_{n_i}f(x)|}{(\log n_i)^{1+\veps/2}}dx\\
&\geq&\delta M\sum_i \frac{|I_i|\tau(1/|I_i|)}{\big(\log(1/|I_i|)\big)^{3+\veps}}.
\end{eqnarray*}
This yields $\sum_i \phi(5|I_i|)\leq\frac{C\|f\|_1}{\delta M}$ (we recall that $\tau$ is increasing), with $C$ another absolute constant and $M>0$ as large as we want.
Hence, $\mathcal H^{\phi}(E_M(\tau,f))\leq\frac{C\|f\|_1}{\delta M} $ (the length of the intervals of the covering can be arbitrarily small).
This in turn implies $\mathcal H^{\phi}(E(\tau,f))=0$, since $E(\tau,f)=\bigcap_{M>0} E_M(\tau,f)$.
\end{proof}

By applying the previous proposition to $\tau(s)=s^\beta$ and $\phi(s)=s^{1-\beta}/\log(s^{-1})^4$, we get:

\begin{corollary}\label{CORAUBRYL1}
For any $f\in L^1(\TT)$ and any $\beta\in[0,1]$, $\dim_{\mathcal H}(E(\beta,f))\leq 1-\beta.$
\end{corollary}

\section{Prevalence of multifractal behaviour}\label{SECPREVALENCE}
\subsection{Strategy}
In all this part, $p$ is a fixed real number such that $1\le p<+\infty$. To
prove that a set $A\subset E$ is Haar-null, the Lebesgue measure on the unit ball of a finite-dimensional subspace $V$
can often play the role of the transverse measure. Precisely, if there exists a finite-dimensional subspace $V$ of $E$
such that, for any $x\in E$, $V\cap (x+A)$ has full Lebesgue-measure,
then $A$ is prevalent. Such a finite-dimensional subspace $V$ is called a \emph{probe} for $A$. Of course, it is the same to prove that for any $x\in E$, $(x+V)\cap A$ has full Lebesgue-measure.

We shall use this property to prove prevalence. More precisely, we shall first prove that,
for a fixed $\beta\in[0,1/p]$, the set of functions $f$ in $L^p(\TT)$ satisfying $\dim_{\mathcal H}\big(E(\beta,f)\big)=1-\beta p$
is prevalent. Then we will conclude because a countable intersection of prevalent sets is prevalent.

\subsection{The construction of saturating functions with disjoint spectra}
In this subsection, $\alpha>1$ is fixed. For $j\geq 1$, we define $J=[j/\alpha]+1$, which is smaller than $j-2$
if $j$ is large enough, say $j\geq j_\alpha$. For $0\leq K\leq 2^J-1$, we define the dyadic intervals 
$$I_{K,j}:=\left[\frac{K}{2^J}-\frac1{2^j};\frac{K}{2^J}+\frac1{2^j}\right].$$
We also define
$$\mathbf I_j:=\bigcup_{K=0}^{2^J-1}I_{K,j}\quad\textrm{ and }\quad\mathbf
I'_j:=\bigcup_{K=0}^{2^J-1}2I_{K,j}.$$
The condition $j\ge j_\alpha$ ensures that the $2I_{K,j}$ do not overlap. 
We finally introduce $D_\alpha$ the set of real numbers in $[0,1]$ which are $\alpha$-approximable by dyadics. 
Namely, $x\in[0,1]$ belongs to $D_\alpha$ if there exist two sequences of
integers $(k_n)_{n\ge 0}$ and $(j_n)_{n\ge 0}$ such that
$$\left|x-\frac{k_n}{2^{j_n}}\right|\leq \frac1{2^{\alpha j_n}}.$$
It is easy to check that $D_\alpha$ is contained in $\displaystyle\limsup_{j\to+\infty}\mathbf I_j$. Indeed, let $x\in D_\alpha$.
One may find $J$ as large as we want and $K$ such that $|x-K/2^J|\leq
1/2^{\alpha J}$.  Let $j$ be an integer such that $J-1=[j/\alpha]$ (such an
integer exists because $\alpha\ge 1$). We get
$$\left|x-\frac{K}{2^J}\right|\leq\frac1{2^j}.$$
Finally, $x\in \mathbf I_j$.
Furthermore, it is well-known
that $\dim_{\mathcal H}(D_\alpha)=1/\alpha$ and even that $\mathcal H^{1/\alpha}(D_\alpha)=+\infty$
(see for instance \cite{BV06} and the mass transference principle). It follows that 
$$\dim_{\mathcal H}\left(\displaystyle\limsup_{j\to+\infty}\mathbf I_j\right)\ge\frac1\alpha.$$

We are going to build finite families of functions which behave badly on each $\mathbf I_j$, and which have disjoint spectra. 
The starting point is a modification of the basic construction of \cite{BH11}.
\begin{lemma}\label{LEMBASIC}
Let $j\geq j_\alpha$ and $J=[j/\alpha]+1$. There exists a trigonometric polynomial $P_j$ with spectrum contained in 
$(0,2^{j+1}-1]$ such that
\begin{itemize}
\item $\|P_j\|_p\leq 1$
\item $|P_j(x)|\geq C 2^{-(J-j)/p}$ for any $x\in \mathbf I_j$
\end{itemize}
where the constant $C$ is independant of $j$.
\end{lemma}
\begin{proof}
Let $\chi_j$ be a continuous piecewise linear function equal to 1 on $\mathbf I_j$,
equal to 0 outside $\mathbf I'_j$ and satisfying $0\leq \chi_j\leq 1$
and $\|\chi_j'\|_\infty\leq 2^j$. $P_j$ is defined by 
$$P_j:=2^{-(J-j+2)/p}e_{2^j}\sigma_{2^j}\chi_j.$$
The $L^p$-norm of $P_j$ is clearly less than or equal to 1 (observe that the measure of $\mathbf I'_J$ is $2^{J-j+2}$).
Applying Lemma 1.7 of \cite{BH11} to $1-\chi_j$, we find 
that $\sigma_{2^j}\chi_j(x)\geq 1/4$ for any $x\in\mathbf I_j$. This gives  
the second assertion of the lemma.
\end{proof}
We now collapse these polynomials to get as many saturating functions as necessary, with disjoint spectra.
\begin{lemma}\label{LEMSFUNCTIONS}
Let $s\geq 1$. There exist functions $g_1,\dots,g_s$ in $L^p(\TT)$ and sequences of integers $(n_{j,r})_{j\geq j_\alpha,1\leq r\leq s}$,
$(m_{j,r})_{j\geq j_\alpha,1\leq r\leq s}$ satisfying
\begin{itemize}
\item $1\leq m_{j,r}<n_{j,r}\leq C2^j$ for any $j$ and any $r$;
\item for any $j\geq j_\alpha$, any $x\in\mathbf I_j$, any $r\in\{1,\dots,s\}$,
$$\left|S_{n_{j,r}}g_r(x)-S_{m_{j,r}}g_r(x)\right|\ge \frac{C}{j^2}2^{(j-J)/p}$$
\item for any $r\in\{1,\dots,s\}$, the spectrum of $g_r$ is included in
  $\bigcup_{j\ge j_\alpha}(m_{j,r},n_{j,r}]=:G_r$
\item if $r_1\not=r_2$, $G_{r_1}\cap G_{r_2}=\emptyset$.
\end{itemize}
\end{lemma}
\begin{proof}
For $r\in\{1,\dots,s\}$, we set
\begin{eqnarray*}
g_r&:=&\sum_{j\geq j_\alpha}\frac1{j^2}e_{(s+r)2^{j+1}}P_j.
\end{eqnarray*}
Define 
\begin{eqnarray*}
m_{j,r}&:=&(s+r)2^{j+1}\\
n_{j,r}&:=&(s+r)2^{j+1}+(2^{j+1}-1)
\end{eqnarray*}
so that each $g_r$ belongs to $L^p$ with spectrum included in $\bigcup_{j\ge
  j_\alpha}(m_{j,r},n_{j,r}]$. Moreover, the intervals $(m_{j,r},n_{j,r}]$ are
disjoint, so that 
$$\left|S_{n_{j,r}}g_r-S_{m_{j,r}}g_r\right|=\frac1{j^2}|P_j|.$$

Let us also remark that, for any $j\ge j_\alpha$ and any $r<s$, $n_{j,r}<m_{j,r+1}$ and
$n_{j,s}<m_{j+1,1}$ so that the spectra $G_1,\cdots , G_s$ are disjoint. This
ends up the proof. 
\end{proof}
It is easy to show that, if $x\in\limsup_j \mathbf I_j$, $r\in\{1,\dots,s\}$
and  $\beta<\frac1p\left(1-\frac1\alpha\right)$, then 
$$\limsup_{n\to +\infty} \frac{|S_n g_r(x)|}{n^\beta}=+\infty.$$
In some sense, the functions $g_r$ have the worst possible behaviour on
$\mathbf I_j$ if we keep in mind that they have to belong to $L^p(\TT)$. We now show that this property remains true almost everywhere (in the sense of the lebesgue measure)  on any affine subspace
$f+\textrm{span}(g_1,\dots,g_s)$ provided $s$ is large enough. This is the main step towards the proof of Theorem \ref{THMPREVALENCE}.
\subsection{Prevalence of divergence for a fixed divergence index}
We keep the notations of the previous subsection.
\begin{proposition}\label{PROPPREVALENCE}
Let $0<\beta<\frac1p\left(1-\frac1\alpha\right)$. There exists $s\geq 1$ such that, for every $f\in L^p(\TT)$, 
for almost every $c=(c_1,\dots,c_s)$ in $\mathbb R^s$, the function 
$g=f+c_1g_1+\dots+c_sg_s$ satisfies for every $x\in D_\alpha$
$$\limsup_{n\to +\infty}\frac{|S_ng(x)|}{n^\beta}=+\infty.$$
\end{proposition}
\begin{proof}
We set $\varepsilon=\frac1p\left(1-\frac1\alpha\right)-\beta$. Let
$s>4/\varepsilon$ and let $f$ be an arbitrary function in $L^p(\TT)$. For
such a value of $s$, we will
prove the conclusion of the proposition for every $x\in\limsup_j \mathbf I_j$
(recall that $D_\alpha\subset\limsup_j \mathbf I_j$).

\smallskip
Let $M>0$ and let us introduce
$$S_{M}:=\left\{g\in L^p(\mathbb T);\ \exists x\in \limsup_{j\to +\infty} \mathbf I_j\textrm{ s.t. }\forall n\geq 1,\ |S_ng(x)|\leq Mn^\beta\right\}.$$
It is enough to show that for every $R>0$, the set of $c\in\mathbb R^s$
satisfying $\|c\|_\infty\le R$ and such that $f+c_1g_1+\dots+c_sg_s$ belongs
to $S_{M}$ has Lebesgue measure 0. In the sequel, we will fix such values of
$M$ and $R$.

\smallskip
If $j\geq 1$, we split each interval $I_{K,j}$ into $2^j$ subintervals. 
Each of them has size $2^{-2j+1}$, and we get $2^{J+j}$ intervals $O_{l,j}$ with
$\bigcup_{l=1}^{2^{J+j}}O_{l,j}=\mathbf I_j$. For $j\geq 1$, $l\in\{1,\dots,2^{J+j}\}$, we set
\begin{eqnarray*}
S_{M}^{(l,j)}&:=&\big\{g\in L^p(\mathbb T);\ \exists x\in O_{l,j}\textrm{ s.t. }
\forall n\geq 1,\ |S_ng(x)|\leq Mn^\beta\big\}.\\
\end{eqnarray*}
Clearly, 
$$S_{M}\subset\limsup_{j\to+\infty}\bigcup_{l=1}^{2^{J+j}}S_{M}^{(l,j)}$$
and we shall first control the size of the $c\in\RR^s$ with $\|c\|_\infty\le
R$ such that 
$$f+c_1g_1+\dots+c_sg_s\in S_{M}^{(l,j)}.$$

We denote by $\lambda_s$ the Lebesgue measure on $\mathbb R^s$ and we fix
$j\geq j_\alpha$, $l$ in $\{1,\dots,2^{J+j}\}$ and $c,c^0$ in $\mathbb R^s$
such that
$\|c\|_\infty\le R$, $\|c^0\|_\infty\le R$ and
$$
\left\{
\begin{array}{rcl}
f+c_1g_1+\dots+c_sg_s&\in&S_{M}^{(l,j)}\\
f+c_1^0 g_1+\dots+c_s^0g_s&\in&S_{M}^{(l,j)}.
\end{array}
\right.
$$
Let $r\in\{1,\dots,s\}$ and let us apply the definition of $S_{M}^{(l,j)}$
with $n=n_{j,r}$ and $n=m_{j,r}$. The spectra $(G_l)_{l\not= r}$ being disjoint from $G_r$, we can find
$x\in O_{l,j}$ such that
$$
\left|S_{n_{j,r}}f(x)-S_{m_{j,r}}f(x)+c_r\big(S_{n_{j,r}}g_r(x)-S_{m_{j,r}}g_r(x)
\big)\right|\leq Mn_{j,r}^\beta+Mm_{j,r}^\beta\leq 2CM2^{\beta j}.
$$
In the same way, we can find $y\in O_{l,j}$ such that 
$$
\left|S_{n_{j,r}}f(y)-S_{m_{j,r}}f(y)+c_r^0\big(S_{n_{j,r}}g_r(y)-S_{m_{j,r}}g_r(y)
\big)\right|\leq 2CM2^{\beta j}.
$$
Using the triangle inequality, we get 
\begin{eqnarray}
\left|c_r\big(S_{n_{j,r}}g_r(x)-S_{m_{j,r}}g_r(x)\big)-c_r^0\big(S_{n_{j,r}}g_r(y)-S_{m_{j,r}}g_r(y)\big)\right|&\leq&\nonumber\\
\quad\quad 4CM2^{\beta j}+|S_{n_{j,r}}f(x)-S_{n_{j,r}}f(y)|+|S_{m_{j,r}}f(x)-S_{m_{j,r}}f(y)|.\label{EQPREVALENCE1}
\end{eqnarray}
Now, by combining the norm of the Riesz projection, Nikolsky's inequality and Bernstein's inequality, 
we know that
$$\|(S_n f)'\|_\infty\leq C(\log n) n^{1+1/p}\|f\|_p$$
(the factor $\log n$ disappears when $p>1$). This yields
\begin{eqnarray*}
|S_{n_{j,r}}f(x)-S_{n_{j,r}}f(y)|&\leq&C\log(n_{j,r})n_{j,r}^{1+1/p}|x-y|\|f\|_p\\
&\leq&Cj2^{j(1+1/p)}2^{-2j+1}\|f\|_p\\&\ll&2^{\beta j}.
\end{eqnarray*}
The same is true for $|S_{m_{j,r}}f(x)-S_{m_{j,r}}f(y)|$ and we get 
\begin{eqnarray}\label{EQPREVALENCE2}
\left|c_r\big(S_{n_{j,r}}g_r(x)-S_{m_{j,r}}g_r(x)\big)-c_r^0\big(S_{n_{j,r}}g_r(y)-S_{m_{j,r}}g_r(y)\big)\right|\le
\kappa 2^{\beta j}
\end{eqnarray}
for some constant $\kappa$ depending on $M$ and $\|f\|_p$ but not on $j$. 

In the same way, 
$$\|(S_ng_r)'\|_\infty\leq C(\log n) n^{1+1/p}\|g_r\|_p\le C(\log n)
n^{1+1/p}.$$
 It follows that
\begin{eqnarray*}
\left|c_r^0\left(\big(S_{n_{j,r}}g_r(x)-S_{m_{j,r}}g_r(x)\big)- \big(S_{n_{j,r}}g_r(y)-S_{m_{j,r}}g_r(y)\big)\right)\right|&\leq&CRj2^{j(1+1/p)}2^{-2j+1}\\&\ll&2^{\beta j}.
\end{eqnarray*}
Combining with (\ref{EQPREVALENCE2}) we obtain a new constant $\kappa$ 
depending on $M$, $\|f\|_p$  and $R$ but not on $j$ such that
\begin{eqnarray}\label{EQPREVALENCE3}
\left|(c_r-c_r^0)\big(S_{n_{j,r}}g_r(x)-S_{m_{j,r}}g_r(x)\big)\right|\le \kappa 2^{\beta j}.
\end{eqnarray} 
Dividing (\ref{EQPREVALENCE3}) by $|S_{n_{j,r}}g_r(x)-S_{m_{j,r}}g_r(x)|$
(which is not equal to zero), we find 
\begin{eqnarray*}
|c_r-c_r^0|&\leq&\kappa 2^{\beta j}|S_{n_{j,r}}g_r(x)-S_{m_{j,r}}g_r(x)|^{-1}\\
&\leq&\frac\kappa{C} 2^{\beta j} j^22^{-(j-J)/p}\\
&\leq&\frac{\kappa 2^{1/p}}{C}j^22^{-\veps j}\\
&\leq& 2^{-\veps j/2}
\end{eqnarray*}
provided $j$ is large enough. Thus, the set of $c\in\mathbb R^s$ with $\|
c\|_\infty\le R$ and such that 
$f+c_1 g_1+\dots+c_s g_s\in S_{M}^{(l,j)}$ is contained in a ball (for the $l^\infty$-norm) of radius $2^{-\veps j/2}$.
Taking the $s$-dimensional Lebesgue measure, this yields
$$\lambda_s\left(\left\{c\in\mathbb R^s;\ \|c\|_\infty\le R\ \mbox{and}\ f+c_1g_1+\dots+c_sg_s\in S_M^{(l,j)}\right\}\right)\leq 2^s 2^{-\veps sj/2}.$$
This in turn gives
$$\lambda_s\left(\left\{c\in\mathbb R^s;\ \|c\|_\infty\le R\ \mbox{and}\ f+c_1g_1+\dots+c_sg_s\in \bigcup_{l=1}^{2^{J+j}}S_{M}^{(l,j)}\right\}\right)\leq 2^s 2^{2j-\veps sj/2}.$$
Thus, since $\veps s/2>2$, this last quantity is the general term of a
convergent series.  
Remember that 
$$S_{M}\subset\limsup_{j\to +\infty} \bigcup_{l=1}^{2^{J+j}} S_{M}^{(l,j)}.$$ The
conclusion of  Proposition \ref{PROPPREVALENCE} follows from Borel Cantelli's lemma.
\end{proof}

\begin{corollary}\label{CORPREVALENCE}
Let $\alpha>1$. For almost every function $f$ in $L^p(\TT)$, for every $x\in D_\alpha$, 
$$\limsup_{n\to +\infty} \frac{\log |S_n f(x)|}{\log n}\geq \frac1p\left(1-\frac1\alpha\right).$$
\end{corollary}
\begin{proof}
This follows immediately from Proposition \ref{PROPPREVALENCE}, taking a sequence $(\beta_n)$ increasing to $\frac1p\left(1-\frac1\alpha\right)$
and using the fact that a countable intersection of prevalent sets remains prevalent.
\end{proof}

\subsection{The general case}
We are now able to complete the proof of Theorem \ref{THMPREVALENCE}, that is 
to prove that 
almost every function $f\in L^p(\TT)$ in the sense of
prevalence has a multifractal behaviour with respect to the summation of 
its Fourier series.
Indeed, let $(\alpha_k)_{k\ge 0}$ be a dense sequence in $(1,+\infty)$. By Corollary \ref{CORPREVALENCE}, for almost every function $f\in L^p(\TT)$,
for every $k\in\mathbb N$ and every $x\in D_{\alpha_k}$, 
$$\limsup_{n\to +\infty} \frac{\log |S_n f(x)|}{\log n}\geq \frac1p\left(1-\frac1{\alpha_k}\right).$$
Now, let $\alpha>1$ and consider a subsequence $(\alpha_{\phi(k)})_{k\ge 0}$ which increases to $\alpha$. Then 
$D_\alpha\subset \bigcap_{k\ge 0} D_{\alpha_{\phi(k)}}$ and for any $x\in D_\alpha$, 
$$\limsup_{n\to +\infty} \frac{\log |S_n f(x)|}{\log n}\geq \frac1p\left(1-\frac{1}{\alpha}\right).$$
The conclusion follows now exactly the argument of \cite{BH11}. For the sake of completeness, we give a
complete account. 
Define
\begin{eqnarray*}
D_\alpha^1&=&\left\{x\in D_\alpha;\ \limsup_{n\to +\infty} \frac{\log |S_nf(x)|}{\log n}=\frac1p\left(1-\frac{1}{\alpha}\right)\right\}\\
D_\alpha^2&=&\left\{x\in D_\alpha;\ \limsup_{n\to +\infty}\frac{\log |S_nf(x)|}{\log n}>\frac1p\left(1-\frac{1}{\alpha}\right)\right\},
\end{eqnarray*}
so that $\mathcal{H}^{1/\alpha}(D_\alpha^1\cup
D_\alpha^2)=\mathcal{H}^{1/\alpha}(D_\alpha)=+\infty$. It suffices 
to prove that $\mathcal{H}^{1/\alpha}(D_\alpha^2)=0$. 
Let $(\beta_n)_{n\ge 0}$ be a sequence of real numbers such that 
$$\beta_n
>\frac1p\left(1-\frac1\alpha\right)\quad \mbox{and}\quad\lim_{n\to
  +\infty}\beta_n=\frac1p\left(1-\frac1\alpha\right).$$

Let us observe that
$$D_\alpha^2\subset \bigcup_{n\ge 0}\mathcal E(\beta_n,f).$$
Moreover, Theorem \ref{THMAUBRY} for $p>1$ and Corollary \ref{CORAUBRYL1} for $p=1$ imply that
$\mathcal{H}^{1/\alpha}(\mathcal E(\beta_n,f))=0$ for
all $n$. Hence, $\mathcal{H}^{1/\alpha}(D_\alpha^2)=0$ and $\mathcal{H}^{1/\alpha}(D_\alpha^1)=+\infty$, which proves
that
$$\dim_\mathcal{H}\left(E\left(\frac1p\left(1-\frac{1}{\alpha}\right),f\right)\right)
\ge\frac1\alpha.$$
By Theorem \ref{THMAUBRY} and Corollary \ref{CORAUBRYL1} again, this
inequality is necessarily an equality. Finally, such a function $f$ satisfies
the conclusion of Theorem \ref{THMPREVALENCE}, setting $1-\beta p=1/\alpha$.

\section{Rapid divergence on big sets for Fourier series of continuous functions}\label{SECCT}
This section is devoted to the proof of Theorem \ref{THMCT} and Theorem \ref{THMCTPREVALENT}. We need to construct functions in $\mathcal C(\TT)$
for which the Fourier series behave badly on a set with Hausdorff dimension 1. We will construct these functions
by blocks. For $k\geq 1$ and $\omega>1$, we set
$$J_k^\omega:=\bigcup_{j=0}^{k-1}\left[\frac jk-\frac1{2\omega k},\frac jk+\frac1{2\omega k}\right]$$
which will be seen as a subset of $\TT=\mathbb R/\mathbb Z$. The construction
makes use of holomorphic functions, so that we will also see $\TT$ as the
boundary of the unit disk $\mathbb D$ and $J_k^\omega$ as a part of $\partial
\mathbb D$.
\begin{lemma}\label{LEMHOLO}
There exist three absolute constants $C_1,C_2,C_3>0$ such that, for any $k\geq 3$, for any $\omega\geq\log k$,
one can find a function $f$ which is holomorphic in a neighbourhood of $\DD$ and which satisfies :
\begin{eqnarray}
\forall z\in \overline{\mathbb D},\quad\Re ef(z)&\geq& \frac{C_1}{\omega k}\label{EQP1}\\
\forall z\in J_k^\omega,\quad |f(z)|&\geq& C_2{\omega }\label{EQP2}\\
\forall z\in \mathbb T,\quad |f(z)|&\leq& C_3{\omega }\label{EQP3}\\
\forall z\in \mathbb T,\quad \left|\frac{f'(z)}{f(z)}\right|&\leq& \omega k\label{EQP4}.
\end{eqnarray}
\end{lemma}
\begin{proof}
We set:
\begin{eqnarray*}
\veps&=&\frac1{\omega k}\\
z_j&=&e^{\frac{2\pi ij}k},\ j=0,\dots,k-1\\
f(z)&=&\frac1k\sum_{j=0}^{k-1}\frac{1+\veps}{1+\veps-\overline{z_j}z}
\end{eqnarray*}
and we claim that $f$ is the function we are looking for.
Indeed, for any $z\in\overline{\DD}$ and any $j\in\{0,\dots,k-1\}$,
\begin{eqnarray}\label{EQBASIC1}
\Re e\left(\frac{1+\veps}{1+\veps-\overline{z_j}z}\right)=\frac{1+\veps}{|1+\veps-\overline{z_j} z|^2}\Re e\big(1+\veps-z_j\overline{z}\big)\geq\frac{1+\veps}{(2+\veps)^2}\times\veps\geq C_1\veps,
\end{eqnarray}
which proves (\ref{EQP1}). To prove (\ref{EQP2}), we may assume that $z=e^{2\pi i\theta}$ with 
$\theta\in\left[\frac{-\veps}{2};\frac{\veps}{2}\right]$. Then 
$$\Re e\left(\frac{1+\veps}{1+\veps-\overline{z_0}z}\right)=\frac{1+\veps}{|1+\veps-z|^2}\Re e\big(1+\veps-z\big)\geq\frac {C_2}\veps.$$
Moreover, (\ref{EQBASIC1}) says that for any $j$, $\Re
e\left(\frac{1+\veps}{1+\veps-\overline{z_j}z}\right)\ge 0$. It follows that 
$$\Re e f(z)\ge \frac{C_2}{k \veps}=C_2\omega.$$

Conversely, we want to control $\sup_{z\in\TT}|f(z)|$. Pick any $z=e^{2\pi i\theta}\in\TT$.
By symmetry, we may and shall assume that $|\theta|\leq\frac1{2k}$. Then we get
$$\left|\frac{1+\veps}{1+\veps-\overline{z_0}z}\right|\leq\frac{C}{\veps}$$
for some constant $C>0$. Now, for any $j\in\{1,\dots,k/4\}$, we can write
\begin{eqnarray*}
|1+\veps-\overline{z_j}z|&\geq& |\Im m(\overline{z_j}z)|\\
&\geq&\sin\left(\frac{2\pi j}k-2\pi\theta\right)\\
&\geq&\frac2\pi\times2\pi \left(\frac jk-\theta\right)\\
&\geq&\frac{4}k\left(j-\frac12\right).
\end{eqnarray*}
Taking the sum, 
$$\left|\sum_{j=1}^{k/4}\frac{1+\veps}{1+\veps-\overline{z_j}z}\right|
\leq \frac{k(1+\veps)}4\sum_{j=1}^{k/4}\frac{1}{j-1/2}\leq Ck\log k$$
(the constant $C$ may change from line to line). In the same way, we have 
$$\left|\sum_{j=3k/4}^{k-1}\frac{1+\veps}{1+\veps-\overline{z_j}z}\right|
\leq Ck\log k.$$
If $j\in[k/4,3k/4]$, we also have $|1+\veps-\overline{z_j}z|\geq C$, so that
$$\left|\sum_{j=k/4}^{3k/4}\frac{1+\veps}{1+\veps-\overline{z_j}z}\right|
\leq Ck.$$
Putting this together, we get 
$$|f(z)|=\left|\frac1k\sum_{j=0}^{k-1}\frac{1+\veps}{1+\veps-\overline{z_j}z}\right|\leq C\left(\frac{1}{k\veps}+\log k+1\right)\leq C_3\omega$$
(this is the place where we need that $\omega\geq\log k$).
Finally, it remains to prove (\ref{EQP4}). We observe that
$$\frac{f'(z)}{f(z)}=\frac{\sum_{j=0}^{k-1}\frac{\overline{z_j}}{(1+\veps-\overline{z_j}z)^2}}{\sum_{j=0}^{k-1}\frac{1}{1+\veps-\overline{z_j}z}}.$$
We deduce that
\begin{eqnarray*}
 \left|\frac{f'(z)}{f(z)}\right|&\leq&\frac{\sum_{j=0}^{k-1}\frac{1}{|1+\veps-\overline{z_j}z|^2}}
 {\sum_{j=0}^{k-1}\frac{\Re e(1+\veps-z_j\bar z)}{|1+\veps-\overline{z_j}z|^2}}\\
 &\leq&\frac{\sum_{j=0}^{k-1}\frac{1}{|1+\veps-\overline{z_j}z|^2}}
 {\sum_{j=0}^{k-1}\frac{\veps}{|1+\veps-\overline{z_j}z|^2}}\\
 &\leq&\frac 1\veps=\omega k.
\end{eqnarray*}
\end{proof}

The crucial step is given by the following lemma.
\begin{lemma}\label{LEMCLECT}
Let $(\veps_n)_{n\ge 1}$ be a sequence of positive real numbers decreasing to
zero. Then, if $n$ is large enough,
one can find an integer $k_n$, a real number $\omega_n>1$ and a trigonometric polynomial $P_n$
with spectrum in $[1,2n-1]$ such that
\begin{itemize}
\item $\|P_n\|_\infty\leq 1$;
\item For any $x\in J_{k_n}^{\omega_n}$, $|S_n P_n(x)|\geq\veps_n \log(n)$.
\end{itemize}
Moreover, we can choose $k_n$ and $\omega_n$ such that $(k_n)$ goes to $+\infty$ and $\omega_n=o(k_n^\alpha)$ for any $\alpha>0$.
\end{lemma}
\begin{proof}
It is clear that the conclusion of the lemma is more difficult to obtain when the sequence $(\veps_n)$ is large.
Thus, we may assume that
$$\veps_n\geq\frac{\log \log n}{4\pi\log n}.$$
In particular, $\veps_n\log n$ goes to infinity. 
We define $k_n$ and $\omega_n$ by
\begin{itemize}
\item $\omega_n$ is equal to $\exp\big(4\pi(\log n)\veps_n\big)$
\item $k_n$ is the biggest integer $k$ satisfying 
$$2\pi k\omega_n\leq n.$$
\end{itemize}
Observe that $\omega_n\ge\log n$ and $\omega_n=o(n^\alpha)$ for all
$\alpha>0$. Then, the inequalities 
$$2\pi k_n\omega_n\leq n\leq 2\pi (k_n+1)\omega_n$$ 
ensure that 
$$k_n\le n\le C k_n n^{1/2}$$
if $n$ is large enough.
It follows that $(k_n)$ goes to $+\infty$, that $\omega_n\geq\log k_n$ and
that $\omega_n=o(k_n^\alpha)$ for any $\alpha>0$.

Let $f_n$ be the holomorphic function given by Lemma \ref{LEMHOLO} for the values $k=k_n$ and $\omega=\omega_n$.
We take $h_n(z)=\log(f_n(z))$, which defines a holomorphic function
in a neighbourhood of $\overline{\DD}$ (remember (\ref{EQP1})). Moreover, $|\Im
m(h_n(z))|\leq\pi/2$ for any $z\in\overline{\DD}$ and $h_n(0)=0$. 
Now, we look at
the function $h_n$ on the boundary of the unit disk $\DD$, that is we introduce
the function $g_n(x)=h_n(e^{2i\pi x})$ defined on the circle 
$\TT=\mathbb R/\mathbb Z$. The properties satisfied by $f_n$ translate into 
\begin{eqnarray*}
\forall x\in J_{k_n}^{\omega_n},\quad |g_n(x)|&\geq&\log \omega_n+\log C_2\\
\forall x\in\TT,\quad |g_n(x)|&\leq&\log \omega_n +\log C_3\\
\forall x\in\TT,\quad |g_n'(x)|&\leq&2\pi k_n\omega_n\leq n.
\end{eqnarray*}
We apply Lemma 1.7 of \cite{BH11}, which is a precised version of F\'ejer's theorem, to the function
$\theta_x(t)=g_n(t)-g_n(x)$ for $x\in\TT$.  
Since $\|\theta_x\|_\infty\leq 2\log\omega_n+2\log C_3$, $\|\theta_x'\|_\infty\leq n$ and $\theta_x(x)=0$, we get
$$|\sigma_n\theta_x(x)|\leq\frac12\log \omega_n +C_4$$
for some absolute constant $C_4$. If $x\in J_{k_n}^{\omega_n}$ we deduce that 
$$|\sigma_n g_n(x)|\geq \frac12\log \omega_n -C_5.$$
Finally we set $$P_n=\frac{2}\pi e_{n}\,\sigma_n(\Im m g_n)=\frac{2}\pi e_{n}\Im m (\sigma_n g_n),$$
so that $\|P_n\|_\infty\leq 1$.
Now, remember that $g_n$ is the restriction to the circle of an holomorphic
function $h_n$ satisfying $h_n(0)=0$. We can then write $\sigma_n
g_n=\sum_{j=1}^{n-1} a_j e_j$, so that $2i\Im m\ \sigma_n
g_n=-\sum_{j=1}^{n-1}\overline{a_j}e_{-j}+\sum_{j=1}^{n-1} a_j e_j$. Thus, the spectrum of $P_n$ is contained in $[1,2n-1]$. Moreover, for any 
$x\in J_{k_n}^{\omega_n}$, we get 
\begin{eqnarray*}
|S_n P_n(x)|&=&\frac 1\pi\left|\sum_{j=1}^{n-1}\overline{a_j}e_{-j+n}\right|\\
&=&\frac 1\pi| \sigma_n g_n(x)|\\
&\geq&\frac1{2\pi}\log \omega_n-C_6\\
&=&2\veps_n\log n-C_6\\
&\geq& \veps_n\log n
\end{eqnarray*}
if $n$ is large enough.
\end{proof}
We are now ready to construct the dense $G_\delta$-set of functions required in Theorem \ref{THMCT}.
\begin{proof}[Proof of Theorem \ref{THMCT}]
Let $(\delta_n)_{n\ge 2}$ be a sequence going to 0. We first consider an
auxiliary sequence $(\delta'_n)_{n\ge 1}$ such that 
$$\lim_{n\to +\infty}\delta'_n= 0,\quad\lim_{n\to
  +\infty}\frac{\delta'_n}{\delta_n}=+\infty\quad\mbox{and}\quad \lim_{n\to
  +\infty}\delta'_n\log n=+\infty\ .$$ 
Let $(g_n)_{n\ge 1}$ be a dense sequence in $\mathcal C(\TT)$, such that the
spectrum of $g_n$  is contained in $[-n,n]$. We set $\eta_n=\max(\delta'_k;\ n\leq k)$. The
sequence $(\eta_n)_{n\ge 1}$ decreases to zero. Moreover, we fix a sequence
$(\veps_n)_{n\ge 1}$, going to zero, such that $\veps_n/\eta_n$ tends to infinity.
Lemma \ref{LEMCLECT} gives us an integer $N$, a sequence $(P_j)_{j\ge N}$ of trigonometric polynomials with spectrum contained in $[1,2j-1]$,
a sequence $(k_j)_{j\ge N}$ of integers going to $+\infty$ and a sequence
$(\omega_{j})_{j\ge N}$ satisfying $\omega_j>1$, such that
$$|S_j P_j(x)|\geq\veps_j\log j$$
for any $x\in J_{k_j}^{\omega_{j}}$. Moreover, we can choose $\omega_j$ such
that $\omega_{j}=o(k_j^\alpha)$ for any $\alpha>0$.

Let us define for $j\ge N$
$$h_j:=g_j+\frac{\eta_j}{\veps_j}e_jP_j.$$
The sequence $(h_j)_{j\ge N}$  remains dense in $\mathcal C(\TT)$. 
Let us also observe that the spectra of $g_j$ and
$\frac{\eta_j}{\veps_j}e_jP_j$ are disjoint. It follows 
that if $x\in J_{k_j}^{\omega_j}$,
$$|S_{2j}h_j(x)-S_j h_j(x)|=\left|\frac{\eta_j}{\veps_j}S_jP_j(x)\right|\ge\eta_j\log j.$$
Thus, for any $x\in J_{k_j}^{\omega_{j}}$, one may find $n\in\{j,2j\}$ such that
$$|S_n h_j(x)|\geq\frac12 \eta_j\log j\geq\frac12\delta'_n(\log n-\log 2).$$
Let $r_j>0$ be small enough so that
$$|S_n h(x)|\geq |S_n h_j(x)|-1$$
for any $h\in B(h_j,r_j)$ and any $n\in\{j,2j\}$ (the open balls are related
to the norm $\|\ \|_\infty$).

Then, we claim that the following dense $G_\delta$-set of $\mathcal C(\TT)$ fulfills all the requirements:
$$G:=\bigcap_{p\ge N}\bigcup_{j\geq p}B(h_j,r_j).$$
Indeed, pick any $h$ in $G$ and any increasing sequence $(j_p)$ such that $h$ belongs to $B(h_{j_p},r_{j_p})$.
Setting $\rho_p=\omega_{j_p}$ and $s_p=k_{j_p}$, it is not hard to show that 
$$E:=\limsup_{p\to+\infty} E_p,\textrm{ with }E_p=J_{s_p}^{\rho_p}$$
has Hausdorff dimension 1. Indeed, remember that for any $\alpha>0$, $\omega_j=o(k_j^\alpha)$. It follows for any $\alpha>0$ and for $p$ large enough, $E_p$ contains
$$F_p=\bigcup_{j=0}^{s_p-1}\left[\frac{j}{s_p}-\frac{1}{2s_p^{1+\alpha}};\frac{j}{s_p}+\frac{1}{2s_p^{1+\alpha}}\right],$$
Now, it is well-known that $\limsup_p F_p$ has Hausdorff dimension equal to $1/(1+\alpha)$ (this follows for instance from
the mass transference principle of \cite{BV06}). Finally, $\dim_{\mathcal H}(E)\ge\frac1{1+\alpha}$. 

Moreover, for any $x\in E$,
the work done before and the fact that $\delta'_n\log n$ goes to $+\infty$ show that 
$$|S_n h(x)|\geq\frac12\delta'_n(\log n-\log 2)-1\ge\frac14\delta'_n\log n$$
for infinitely many values of $n$. We then get
$$\frac{|S_n h(x)|}{\delta_n\log n}\geq\frac{\delta'_n}{4\delta_n}$$
for infinitely many values of $n$. 
This achieves the proof of Theorem \ref{THMCT}.
\end{proof}
We can finally construct the prevalent set of functions required in Theorem \ref{THMCTPREVALENT}. 
\begin{proof}[Proof of Theorem \ref{THMCTPREVALENT}]
Let $(\delta_n)_{n\ge 2}$ be a sequence going to 0 and denote by $A$ the set of continuous  functions $f\in\mathcal C(\TT)$ such that
$$\dim_{\mathcal H}\left(\left\{ x\in\TT\ ;\ \limsup_{n\to+\infty}\frac{|S_nf(x)|}{\delta_n\log n}=+\infty\right\}\right)<1.$$
We have to prove that $A$ is Haar-null in $\mathcal C(\TT)$.

Let $f_0$ be a fixed function in the complementary of $A$ (such a function does exist by Theorem \ref{THMCT}) and let $g$ be an arbitrary function in $\mathcal C(\TT)$. Suppose that $t_1$ and $t_2$ are two real numbers such that
$$t_1f_0\in(g+A)\quad\mbox{ and }\quad t_2f_0\in(g+A).$$
We can then find $f_1\in A$ and $f_2\in A$ such that $(t_1-t_2)f_0=f_1-f_2$. It is clear that $f_1-f_2\in A$ ($A$ is a vector subspace of $\mathcal C(\TT)$). It follows that $t_1=t_2$, so that
$$\#\left(\mbox{span}\,(f_0)\cap(g+A)\right)\le 1.$$
In particular, the Lebesgue-measure in $\mbox{span}\,(f_0)$ is transverse to $A$ and $A$ is Haar-null in  $\mathcal C(\TT)$.
\end{proof}
\begin{rem}
 We have just only proved that a proper subspace in a complete metric vector space is Haar-null. This property is probably well-known.
\end{rem}

\end{document}